\documentclass[12pt]{amsart}

\usepackage{a4}
\usepackage{amsmath, amssymb}
\usepackage{mathrsfs}
\usepackage{exscale}
\usepackage[varg]{txfonts}
\usepackage{color}
\usepackage{enumerate}
\usepackage{xy}
\xyoption{all}
\renewcommand{\implies}{\Rightarrow}

\newcommand{\too}{\longrightarrow}

\newcommand{\ox}{\otimes}

\newcommand{\ch}{\mathrm{char}}

\newcommand{\<}{\langle}
\renewcommand{\>}{\rangle}

\DeclareMathOperator*{\lperp}{\Huge \textrm {$\perp$}}

\newcommand{\N}{\mathbb{N}}

\newcommand{\vf}{\varphi}

\newcommand{\gL}{\Lambda}
\newcommand{\Kd}{{K^{\times}}}

\newcommand{\HP}{\mathbb{H}}

\newcommand{\nslo}{\frac{n^2-1}{2}}

\newcommand{\nslfo}{\frac{n^2-4}{2}}

\newcommand{\kpow}{\gL^k}
\newcommand{\hyp}{\mathrm {Hyp}}
\newcommand{\lb}{\left(}
\newcommand{\rb}{\right)}

\newtheorem{lemma}{Lemma}[section]

\newtheorem{propo}[lemma]{Proposition}

\theoremstyle{definition}

\newtheorem{defi}[lemma]{Definition}

\newtheorem{remark}[lemma]{Remark}

\allowdisplaybreaks[2]

\DeclareFontFamily{OT1}{manual}{}
\DeclareFontShape{OT1}{manual}{m}{n}{ <10> manfnt }{}

\title{Symmetric Powers of Hyperbolic Forms \\ \ and of Trace Forms on Symbol Algebras}
\author{Ronan Flatley}
\date{\today}             
\address{Department of Mathematics and Computer Studies, Mary Immaculate College, Limerick, Ireland}
\email{ronan.flatley@mic.ul.ie}
\subjclass[2010]{Primary: 11E81; Secondary: 16K20}
\keywords{hyperbolic form, symmetric power, symbol algebra, trace form}

\begin{document}
\maketitle

\begin{abstract} Let $K$ be a field with characteristic different from $2$ and let $S$ be a symbol algebra over $K$. We compute the symmetric powers of hyperbolic quadratic forms over $K$. Also, we compute the symmetric powers of the quadratic trace form of $S$. In both cases we apply a generalised form of the Vandermonde convolution in the course of the computations. 
 \end{abstract}

\section{Introduction}\label{sec.int}

We assume throughout this paper that $K$ is a field with characteristic different from $2$. In \cite{F} we computed the exterior powers of hyperbolic forms over $K$ and used these results to compute the exterior powers of the quadratic trace form of symbol algebras over $K$. Here we are interested in finding corresponding results for the symmetric powers of such forms. 

In \cite{M2}, McGarraghy investigated symmetric powers of classes of symmetric bilinear forms in the Witt-Grothendieck ring of $K$, deriving their basic properties as well as computing several of their classical invariants such as determinant, signature with respect to an ordering and Hasse invariant. Their relation to exterior powers was also given. For a given symmetric bilinear form $\varphi$ over $K$ there was a distinction drawn between factorial symmetric powers ${\bf S}^k\varphi$  and non-factorial symmetric powers $S^k\varphi$. Factorial symmetric powers were seen to have the advantages of being computable for any given symmetric bilinear form and of being independent of basis. However, effectively they are useful only in characteristic $0$. On the other hand, non-factorial symmetric powers are independent of characteristic but require the form to be diagonalised, i.e., a basis needs to be chosen. In this paper we deal with non-factorial symmetric powers of forms, hereafter referred to simply as symmetric powers. Since $\ch{(K)}\neq 2$, we use the one-to-one correspondence between symmetric bilinear forms and quadratic forms.

Notation and terminology is borrowed from Lam \cite{TYL} and Scharlau \cite{S}. A diagonalised quadratic form over $K$ with coefficients $a_1, \dots, a_m \in \Kd$ is denoted by $\<a_1, \dots, a_m\>$. The hyperbolic plane $\<1, -1\>$ is denoted by $\HP$.
If $\vf$ and $\psi$ are forms over $K$ then $\vf \simeq \psi$ means that these forms are isometric. The tensor product of the $1$-dimensional form $\<n\>$ with $\vf$ is denoted $\<n\>\vf$ to be distinguished from $n$ copies of $\vf$ which is written as $n \times \vf$.

\section{Preliminaries}\label{sec.prelim}

\subsection{Exterior powers}
Bourbaki defined the concept of exterior power of a symmetric bilinear form in \cite[IX, \S1, (37)]{B}:

\begin{defi}\label{D3}
Let $V$ be a vector space of dimension $m$ over $K$.
Let $\vf: V \times V \rightarrow K$ be a non-singular bilinear form and let $k$ be a positive integer, $k\le m$. We define the {\it $k$-fold exterior power} of $\vf$, 
$$\kpow \vf: \kpow V \times \kpow V \too K$$
by 
$$\kpow \vf (x_1 \wedge \dots \wedge x_k, y_1 \wedge \dots \wedge y_k) = \det\big(\vf(x_i, y_j)\big)_{1 \le i, j \le k}.$$
The bilinear extension of this form defines $\kpow \vf$ everywhere on $\kpow V \times \kpow V$. 
We define $ \gL^0\vf := \<1\>$. Also, for $k>m$, we define  $\gL^k\vf := 0 $, the zero form.
\end{defi}

We use the following two results on exterior powers of hyperbolic forms:

\begin{propo}\cite[Proposition 5.6]{F}\label{S4}
Let $\phi \simeq h \times \HP$ where $ h \in \mathbb{N}$ and $ k$ odd with $1 \leq k \leq 2h-1$. Then
\[
\gL^k\phi \simeq \frac{1}{2} {2h \choose k} \times\HP.
\]
\end{propo}

\begin{propo}\cite[Proposition 5.7]{F}\label{S5}
Let $\phi \simeq h \times \HP$ where $ h \in \mathbb{N}$, $k=2\ell$ and $0 \le \ell \le h$. Then 
\[
\kpow\phi = \gL^{2\ell}\phi \simeq  {h \choose \ell} \times \<(-1)^\ell\> \perp \frac{1}{2} \lb {2h \choose 2\ell}-{h \choose \ell}\rb \times\HP.
\]
\end{propo}

\subsection{Symmetric powers}

\begin{defi}\cite[adapted statement of Proposition 3.5]{M2}\label{S1}
Let $V$ be a vector space of dimension $m$ over $K$. Let $\vf$ be a symmetric bilinear form over $K$ with $\vf \simeq \<a_1, \dots , a_m\>$. Let $k$ be a non-negative integer. We define the \emph{$k$-fold symmetric power} of $\varphi$,
\[
 S^k\varphi : S^k V \times S^k V \to K
\]
by
\[
 S^k \vf = {\lperp_{\substack{1 \le i_1 < \dots < i_\ell \le m\\ k_{i_1}+ \dots +  k_{i_\ell} =k}}}\<a_{i_1}^{k_{i_1}} \dots a_{i_\ell}^{k_{i_\ell}}\>.
\]
\end{defi}

\begin{remark}
The form $S^k\varphi$ has dimension ${m+k-1}\choose{k}$.
\end{remark}

\begin{remark}
Clearly, $S^0\varphi=\<1\>$ and $S^1\varphi=\varphi$.
\end{remark}

\begin{remark}
We note that $S^k$ preserves isometries, i.e., $\varphi\simeq\psi\implies S^k\varphi\simeq S^k\psi$.
In general the converse is not true; consider a field where $\<1\>\not\simeq\<-1\>$.
Let $\varphi=\<1,1\>\perp\HP$ and $\psi=\<-1,-1\>\perp\HP$. 
Then $\varphi\not\simeq\psi$ but $S^2\varphi\simeq S^2\psi=4\times\<1\>\perp 3\times\HP$.
\end{remark}

\begin{remark}
The $k$-fold exterior power is always a subform of the $k$-fold symmetric power.
Therefore, $\gL^k\varphi$ isotropic $\implies S^k\varphi$ isotropic.
\end{remark}

\begin{lemma}\label{L1}
$S^k (m \times \<1\>) ={{m+k-1} \choose k} \times \<1\>$ and
$S^k (m \times \<-1\>) ={{m+k-1} \choose k} \times \<(-1)^k\>$.
\end{lemma}

The next two propositions give properties that are very useful for the computations in Sections \ref{sec.hp} and \ref{sec.sa}.

\begin{propo}\cite[p. 47]{K}\label{S2}
Let $\varphi$ and $\psi$ be symmetric bilinear forms over $K$ and let $k$ be a positive integer. Then 
\[
   S^k(\varphi\perp\psi)=\lperp_{i+j=k}S^i\varphi\ox S^j\psi\ .
\]
\end{propo}

\begin{propo}\cite[Proposition 4.4]{M2}\label{S3}
Let $\varphi$ be an $n$-dimensional symmetric bilinear form over $K$ and let $k$ be a positive integer. Then 
\[
   S^k \varphi=\lperp_{i=0}^{\lceil{k/2\rceil}} {{n+i-1}\choose{i}}\times \gL^{k-2i}\varphi\ .
\]
\end{propo}

\section{Symmetric powers of hyperbolic forms}\label{sec.hp}

For each dimension of even parity, there exists exactly one hyperbolic form up to isometry. In other words, if $\varphi$  is a hyperbolic form of dimension $2h$, then $\varphi\simeq h\times\HP$. 
In this section we compute the symmetric powers of these forms. For computational convenience,  we often do not state the specific number of hyperbolic planes. Instead we shall use the word ``Hyp" to represent the number of hyperbolic planes required in order to give full dimension to the particular symmetric power in question.

\begin{propo}\label{N1}
Let $k$ be any positive odd integer, $h\in\N$. Then
\[
S^k(h\times\HP)	=\frac{1}{2}{{2h+k-1}\choose{k}}\times\HP.
\]
\end{propo}

\begin{proof}
\begin{align*}
 S^k(h\times\HP)&=S^k(h\times\<1\>\perp h\times\<-1\>)\\
  &= \lperp_{i+j=k}S^i(h\times\<1\>)\ox S^j(h\times\<-1\>)\\
  &= \lperp_{\substack{i+j=k\\j\text{ even}}}S^i(h\times\<1\>)\ox S^j(h\times\<-1\>)  \perp
         \lperp_{\substack{i+j=k\\j\text{ odd}}}S^i(h\times\<1\>)\ox S^j(h\times\<-1\>) 
         \intertext{\hspace{10cm} [by Proposition \ref{S2}]}
  &= \lperp_{\substack{i+j=k\\j\text{ even}}}\left({{h+i-1}\choose{i}}\times\<1\>\right)\ox \left({{h+j-1}\choose{j}}\times\<1\>\right)  \perp
     \lperp_{\substack{i+j=k\\j\text{ odd}}}\left({{h+i-1}\choose{i}}\times\<1\>\right)\ox \left({{h+j-1}\choose{j}}\times\<-1\>\right) \intertext{\hspace{10cm}[by Lemma \ref{L1}]}
  &= \sum_{\substack{i+j=k\\j\text{ even}}}{{h+i-1}\choose{i}}{{h+j-1}\choose{j}}\times\<1\> \perp
        \sum_{\substack{i+j=k\\j\text{ odd}}}{{h+i-1}\choose{i}}{{h+j-1}\choose{j}}\times\<-1\> \\
  &= \sum_{\substack{i+j=k\\j\text{ even}}}{{h+i-1}\choose{i}}{{h+j-1}\choose{j}}\times\<1\> \perp
        \sum_{\substack{i+j=k\\i\text{ even}}}{{h+j-1}\choose{j}}{{h+i-1}\choose{i}}\times\<-1\>\ .
\end{align*}
From the expression above we observe that there is the same number of summands $\<1\>$ as summands $\<-1\>$.
Given that the dimension of $S^k(h\times\HP)$ is ${{2h+k-1}\choose{k}}$, we conclude that
\[
 S^k(h\times\HP)=\frac{1}{2}{{2h+k-1}\choose{k}}\times\HP.
\]
\end{proof}

We now seek to compute $S^k(h\times\HP)$ for $k$ an arbitrary even natural number. First, we need two lemmata:

\begin{lemma}\label{L2}
$$\sum_{j=0}^r{(-1)^{j}}{{2p+j-1}\choose{j}}{{p}\choose{r-j}}=(-1)^{r}{{p+r-1}\choose{r}}\ .$$
\end{lemma}

\begin{proof}
By the Vandermonde convolution (see \cite{G} for example), for any integer $p$ and any positive integer $r$,
\[
 \sum_{j=0}^{r}{{-2p}\choose{j}}{{p}\choose{r-j}}={{-p}\choose{r}}\ .
\] 
By applying the ``minus transformation" ${{-q}\choose{i}}=(-1)^{i}{{q+i-1}\choose{i}}$ to both sides, we get the required identity.
\end{proof}

\begin{lemma}{\emph{[Pascal's rule]}\\}\label{L3}
For any $r\in \N$, $1\le s \le r$, we have that
\[
   {{r-1}\choose s}+{{r-1}\choose {s-1}}={r\choose s}\ .
\]
\end{lemma}

\begin{remark}\label{R1}
In several of the remaining proofs we use the following property of binomial coefficients, which is easy to derive:

\begin{align}
&{{r}\choose {s}} = \frac{r}{s}{{r-1}\choose{s-1}}\ . \notag 
\end{align} 
\end{remark}

\begin{propo}\label{N2}
Let $k$ be any non-negative even integer, $k=2\ell$, $h\in\N$. Then
\[
S^k(h\times\HP)=S^{2\ell}(h\times\HP)={{h+\ell-1} \choose \ell} \times \<1\> \perp \frac{1}{2}\left({{2h+2\ell-1}\choose{2\ell}}-{{h+\ell-1}\choose{\ell}}\right)\times\HP.
\]
\end{propo}

\begin{proof}
Using the formula
\[
S^{k}\phi=\lperp_{i=0}^{\lceil{k/2\rceil}}{{n+i-1}\choose{i}}\times{\gL^{k-2i}\vf}\qquad \text{[Proposition \ref{S3}]}
\]
we get
\begin{align*}
S^k(h\times\HP)&=S^{2\ell}(h\times\HP)\\
&=\lperp_{i=0}^{\ell}{{2h+i-1} \choose {i}} \times \gL^{2\ell-2i}(h\times\HP)\\
&=\lperp_{i=0}^{\ell}{{2h+i-1} \choose {i}} \times \left({{h}\choose{\ell-i}}\times \<(-1)^{\ell-i}\> \perp \frac{1}{2}\left({{2h}\choose{2(\ell-i)}}-{{h}\choose{\ell-i}}\right)\times\HP\right) \intertext{\hspace{10 cm}[by Proposition \ref{S5}}
&=\lperp_{i=0}^{\ell}{{2h+i-1} \choose {i}}{{h}\choose{\ell-i}}\times \<(-1)^{\ell-i}\> \perp \hyp\\
&=\lperp_{i=0}^{\ell}{{2h+i-1} \choose {i}}{{h}\choose{\ell-i}}\times \<(-1)^{i-\ell}\> \perp \hyp\\
&=\sum_{i=0}^{\ell}(-1)^{i}{{2h+i-1} \choose {i}}{{h}\choose{\ell-i}}\times \<(-1)^{-\ell}\> \perp \hyp\\
&=(-1)^{\ell}{{h+\ell-1} \choose {\ell}}\times \<(-1)^{-\ell}\> \perp \hyp  \hspace{3cm } \text{[by Lemma \ref{L2}]}\\
&={{h+\ell-1} \choose {\ell}}\times \<1\> \perp \hyp\ .
\end{align*}
\end{proof}

\section{Symmetric powers of trace forms on symbol algebras}\label{sec.sa}

Let $n$ be an arbitrary positive integer and let $K$ contain a primitive $n$-th root of unity $\omega$. Let $a$, $b \in$ $\Kd$ and 
 let $S$ be the algebra over $K$ generated by elements $x$ 
 and $y$ where $$x^n = a, \quad y^n=b \quad\text{ and } \quad yx=\omega x y.$$ 
 We call this algebra a {\it symbol algebra} (see \cite[Chapter $1$, \S$2$]{KMRT}) and denote it as $(a, b; n, K, \omega)$. In \cite[\S11]{D}, Draxl shows it to be a central simple algebra over $K$ of degree $n$. 

\par Let $A$ be a central simple algebra of degree $n$ over a field $K$ of characteristic different from $2$. We write $T_A \colon A \rightarrow K$ for the quadratic trace form $$T_A(z)={\text {Trd}}_A(z^2) {\quad\text{for }} z \in A,$$ where Trd$_A$ is the reduced trace of $A$. 

In \cite{F}, we proved the following:
\begin{propo}\cite[Propositions 2.1, 3.1]{F}\label{P1}
For $S=(a,b;n, K, \omega)$, we have
\begin{align}
\text{(i)\quad} T_S &\simeq \<(-1)^{n/2}\> \perp \nslo\times\HP &&\text{ for } n \text{ odd.}\notag\\
\text{(ii)\quad}T_S &\simeq \<n\>\<1,a,b,(-1)^{n/2}ab\> \perp \nslfo \times \HP &&\text{ for } n \text{ even.}\notag
\end{align}
\end{propo}

\begin{propo}\label{P10}
Let $n$ be odd and $k\ge 0$. Then
\[
	S^k T_S= 
	\begin{cases}
		{{\frac{n^2+k-2}{2}} \choose {\frac{k-1}{2}}} \times \<(-1)^\frac{n-1}{2}\> \perp \hyp, &\text{if $k$ is odd;}\\
		{{\frac{n^2+k-1}{2}} \choose \frac{k}{2}}\times \<1\> \perp \hyp, &\text{if $k$ is even.}
	\end{cases}
\]
\end{propo}

\begin{proof}
Let $k$ be odd. Then
\begin{align*}
S^k{T_S}&=S^k\left(\<(-1)^{\frac{n-1}{2}}\>\perp \frac{n^2-1}{2}\times\HP\right) \hspace{4.4cm } \text{[by Proposition \ref{P1}]}\\
&=\lperp_{i+j=k}S^{i}\left(\<(-1)^{\frac{n-1}{2}}\>\right)\ox S^j\left(\frac{n^2-1}{2}\times\HP\right)\hspace{3cm } \text{[by Proposition \ref{S2}]}\\
&=\lperp_{j \text{ even}}S^{i}\left(\<(-1)^{\frac{n-1}{2}}\>\right)\ox S^j\left(\frac{n^2-1}{2}\times\HP\right) \perp
\lperp_{j \text{ odd}}S^{i}\left(\<(-1)^{\frac{n-1}{2}}\>\right)\ox S^j\left(\frac{n^2-1}{2}\times\HP\right)\\
&=\lperp_{j \text{ even}}S^{i}\left(\<(-1)^{\frac{n-1}{2}}\>\right)\ox S^j\left(\frac{n^2-1}{2}\times\HP\right) \perp
\hyp\hspace{1.8cm } \text{[by Proposition \ref{N1}]}\\
&=\lperp_{j \text{ even}}\left(\<(-1)^{\frac{(n-1)i}{2}}\>\ox S^j\left(\frac{n^2-1}{2}\times\HP\right)\right) \perp
\hyp\hspace{2cm } \text{[by Lemma \ref{L1}]}\\
&=\lperp_{j \text{ even}}\left(\<(-1)^{\frac{(n-1)}{2}}\>\ox \left({{\frac{n^2-1+j-2}{2}}\choose{\frac{j}{2}}}\times\<1\>\perp
\left({{n^2-1+j-2}\choose{j}}-{{\frac{n^2-1+j-2}{2}}\choose{\frac{j}{2}}}\right)\times\HP\right)\right) \perp
\hyp \intertext{\hspace{10 cm}[by Proposition \ref{N2}]}
&=\sum_{j \text{ even}}{{\frac{n^2+j-3}{2}}\choose{\frac{j}{2}}}\times\<(-1)^{\frac{n-1}{2}}\>
\perp \hyp\\
&=\sum_{j'=0}^{\frac{k-1}{2}}{{\frac{n^2-1}{2}+j'-1}\choose{j'}}\times\<(-1)^{\frac{n-1}{2}}\>
\perp \hyp \hspace{2.3cm } \text{[changing the index, $j'=j/2$]}\\
&={{\frac{n^2+k-2}{2}}\choose{\frac{k-1}{2}}}\times\<(-1)^{\frac{n-1}{2}}\>
\perp \hyp\ . \hspace{3.8cm } \text{[by repeated use of Lemma \ref{L3}]}
\end{align*}
Now let $k$ be even. Then
\begin{align*}
S^k{T_S}&=S^k\left(\<(-1)^{\frac{n-1}{2}}\>\perp \frac{n^2-1}{2}\times\HP\right) \hspace{4.4cm } \text{[by Proposition \ref{P1}]}\\
&=\lperp_{i+j=k}S^{i}\left(\<(-1)^{\frac{n-1}{2}}\>\right)\ox S^j\left(\frac{n^2-1}{2}\times\HP\right)\hspace{3cm } \text{[by Proposition \ref{S2}]}\\
&=\lperp_{j \text{ even}}S^{i}\left(\<(-1)^{\frac{n-1}{2}}\>\right)\ox S^j\left(\frac{n^2-1}{2}\times\HP\right) \perp
\lperp_{j \text{ odd}}S^{i}\left(\<(-1)^{\frac{n-1}{2}}\>\right)\ox S^j\left(\frac{n^2-1}{2}\times\HP\right)\\
&=\lperp_{j \text{ even}}S^{i}\left(\<(-1)^{\frac{n-1}{2}}\>\right)\ox S^j\left(\frac{n^2-1}{2}\times\HP\right) \perp
\hyp\hspace{1.8cm } \text{[by Proposition \ref{N1}]}\\
&=\lperp_{j \text{ even}}\left(\<(-1)^{\frac{(n-1)i}{2}}\>\ox S^j\left(\frac{n^2-1}{2}\times\HP\right)\right) \perp
\hyp\hspace{2cm } \text{[by Lemma \ref{L1}]}\\
&=\lperp_{j \text{ even}}\left(\<1\>\ox \left({{\frac{n^2-1+j-2}{2}}\choose{\frac{j}{2}}}\times\<1\>\perp
\left({{n^2-1+j-2}\choose{j}}-{{\frac{n^2-1+j-2}{2}}\choose{\frac{j}{2}}}\right)\times\HP\right)\right) \perp
\hyp \intertext{\hspace{10 cm}[by Proposition \ref{N2}]}
&=\sum_{j \text{ even}}{{\frac{n^2+j-3}{2}}\choose{\frac{j}{2}}}\times\<1\>
\perp \hyp\\
&=\sum_{j'=0}^{\frac{k}{2}}{{\frac{n^2-1}{2}+j'-1}\choose{j'}}\times\<1\>
\perp \hyp \hspace{2.3cm } \text{[changing the index, $j'=j/2$]}\\
&={{\frac{n^2+k-1}{2}}\choose{\frac{k}{2}}}\times\<1\>
\perp \hyp\ . \hspace{3.8cm } \text{[by repeated use of Lemma \ref{L3}]}
\end{align*}
\end{proof}

\begin{propo}\label{P11}
Let $n$ be even, $k$ odd. We write $T_S \simeq q_S \perp m$ $\times$ $\HP$ where $q_S \simeq \<n\>\<1, a, b, (-1)^\frac{n}{2}ab\>$ and $m=\nslfo$. Then for $k>1$,
\[
	S^k T_S = 
		\begin{cases}
		\frac{2m+2k+4}{k-1} {{(2m+k+3)/2}\choose{(k-3)/2}}\times q_S
                     \perp \hyp, &\text{if $n\equiv 0\pmod{4}$;}\\
		\frac{2m+6}{k-1} {{(2m+k+3)/2}\choose{(k-3)/2}}\times q_S
                     \perp\hyp, &\text{if $n\equiv 2\pmod{4}$.}
	     \end{cases}
\]
\end{propo}

\begin{proof}
\begin{align*}
S^k{T_S}&=S^k\left(q_S \perp m\times\HP\right)\\
&=\lperp_{i+j=k}S^{i}q_S\ox S^j\left(m\times\HP\right)\hspace{3cm } \text{[by Proposition \ref{S2}]}\\
&=\lperp_{j \text{ even}}S^{i}q_S\ox S^j\left(m\times\HP\right) \perp
\lperp_{j \text{ odd}}S^{i}q_S\ox S^j\left(m\times\HP\right)\\
&=\lperp_{j \text{ even}}S^{i}q_S\ox S^j\left(m\times\HP\right) \perp
\hyp\hspace{1.8cm } \text{[by Proposition \ref{N1}]}\\
&=\lperp_{j \text{ even}}S^{i}q_S\ox \left({{m+\frac{j}{2}-1}\choose{\frac{j}{2}}}\times\<1\>
\perp \frac{1}{2}\left({{2m+j-1}\choose{j}}-{{m+\frac{j}{2}-1}\choose{\frac{j}{2}}}\right)\times \HP \right) \perp
\hyp\intertext{\hspace{8.3cm}[by Proposition \ref{N2}]}
&=\lperp_{j \text{ even}}S^{i}q_S\ox \left({{m+\frac{j}{2}-1}\choose{\frac{j}{2}}}\times\<1\>\right) 
\perp \hyp\\
&=\sum_{j \text{ even}}{{m+\frac{j}{2}-1}\choose{\frac{j}{2}}}\times S^i q_S 
\perp \hyp\\
&=\sum_{j \text{ even}}{{m+\frac{j}{2}-1}\choose{\frac{j}{2}}}\times 
\left(\lperp_{\ell=0}^{(i+1)/2}{{3+\ell}\choose{\ell}}\times \gL^{i-2\ell}q_S\right) 
\perp\hyp \hspace{1cm } \text{[by Proposition \ref{S3}]}\\
&=\sum_{j \text{ even}}{{m+\frac{j}{2}-1}\choose{\frac{j}{2}}}\times 
\left({{3+\frac{i-1}{2}}\choose{\frac{i-1}{2}}}\times \gL^{1}q_S \perp  {{3+\frac{i-3}{2}}\choose{\frac{i-3}{2}}}\times \gL^{3}q_S   \right) \perp\hyp \hspace{1mm} \text{[since for larger $k$, $\gL^k$ is zero]} \\
&=\sum_{j \text{ even}}{{m+\frac{j}{2}-1}\choose{\frac{j}{2}}}\times 
\left({{3+\frac{i-1}{2}}\choose{\frac{i-1}{2}}}\times q_S 
\perp  {{3+\frac{i-3}{2}}\choose{\frac{i-3}{2}}}\times \left(\<(-1)^{n/2}\> \ox q_S\right)   \right) \perp\hyp \\
&=\sum_{j \text{ even}}{{m+\frac{j}{2}-1}\choose{\frac{j}{2}}}
{{3+\frac{i-1}{2}}\choose{\frac{i-1}{2}}}\times q_S 
\perp \sum_{j \text{ even}}{{m+\frac{j}{2}-1}\choose{\frac{j}{2}}}
{{3+\frac{i-3}{2}}\choose{\frac{i-3}{2}}}\times \left(\<(-1)^{n/2}\>\ox q_S\right)  
\perp\hyp \\
&=\sum_{j \text{ even}}{{m+\frac{j}{2}-1}\choose{\frac{j}{2}}}
{{3+\frac{k-1}{2}-\frac{j}{2}}\choose{\frac{k-1}{2}-\frac{j}{2}}}\times q_S 
\perp \sum_{j \text{ even}}{{m+\frac{j}{2}-1}\choose{\frac{j}{2}}}
{{3+\frac{k-3}{2}-\frac{j}{2}}\choose{\frac{k-3}{2}-\frac{j}{2}}}\times \left(\<(-1)^{n/2}\>\ox q_S\right)  
\perp\hyp \\
&=\sum_{j'=0}^{\frac{k-1}{2}}{{m+j'-1}\choose{j'}}
{{3+\frac{k-1}{2}-j'}\choose{\frac{k-1}{2}-j'}}\times q_S 
\perp \sum_{j \text{ even}}{{m+j'-1}\choose{j'}}
{{3+\frac{k-3}{2}-j'}\choose{\frac{k-3}{2}-j'}}\times \left(\<(-1)^{n/2}\>\ox q_S\right)  
\perp\hyp  \intertext{\hspace{9cm }[changing the index, $j'=j/2$]}
&={{m+3+\frac{k-1}{2}}\choose{\frac{k-1}{2}}}\times q_S 
\perp 
{{m+3+\frac{k-3}{2}}\choose{\frac{k-3}{2}}}\times \left(\<(-1)^{n/2}\>\ox q_S\right)  
\perp\hyp\ . \hspace{1cm}\text{[by \cite[Equation (3)]{G}]}
\end{align*}
Now if $n\equiv 0\pmod 4$
\begin{align*}
S^k T_S &=\left({{m+3+\frac{k-1}{2}}\choose{\frac{k-1}{2}}} +
{{m+3+\frac{k-3}{2}}\choose{\frac{k-3}{2}}}\right)\times q_S
\perp\hyp\ \\  
&=\frac{2m+2k+4}{k-1}
{\frac{2m+k+3}{2}\choose{\frac{k-3}{2}}}\times q_S
\perp\hyp\ . \hspace{3cm} \text{[by Remark \ref{R1}]}
\end{align*}
On the other hand, if $n\equiv 2\pmod 4$
\begin{align*}
S^k T_S &=\left({{m+3+\frac{k-1}{2}}\choose{\frac{k-1}{2}}} -
{{m+3+\frac{k-3}{2}}\choose{\frac{k-3}{2}}}\right)\times q_S
\perp\hyp\ \\  
&=\frac{2m+6}{k-1}
{\frac{2m+k+3}{2}\choose{\frac{k-3}{2}}}\times q_S
\perp\hyp\ . \hspace{5cm} \text{[by Remark \ref{R1}]}
\end{align*}
\end{proof}

\begin{propo}\label{P12}
Let $n$ and $k$ be even. We write $T_S \simeq q_S \perp m$ $\times$ $\HP$ where $q_S \simeq \<n\>\<1, a, b, (-1)^\frac{n}{2}ab\>$ and $m=\nslfo$. Then for $k>1$,
\[
	S^k T_S = 
		\begin{cases}
		\left(1+\frac{(m+3+k/2)(m+2+k/2)}{(k/2)(k/2-1)}\right){{m+1+k/2}\choose{k/2-2}}\times \<1\> 
                     \perp \hyp, &\text{if $n\equiv 0\pmod{4}$;}\\
		\left(\frac{(m+3+k/2)(m+2+k/2)}{(k/2)(k/2-1)}-1\right){{m+1+k/2}\choose{k/2-2}}\times \<1\>
                     \perp\hyp, &\text{if $n\equiv 2\pmod{4}$.}
	     \end{cases}
\]
\end{propo}

\begin{proof}
\begin{align*}
S^k{T_S}&=S^k\left(q_S \perp m\times\HP\right)\\
&=\lperp_{i+j=k}S^{i}q_S\ox S^j\left(m\times\HP\right)\hspace{3cm } \text{[by Proposition \ref{S2}]}\\
&=\lperp_{j \text{ even}}S^{i}q_S\ox S^j\left(m\times\HP\right) \perp
\lperp_{j \text{ odd}}S^{i}q_S\ox S^j\left(m\times\HP\right)\\
&=\lperp_{j \text{ even}}S^{i}q_S\ox S^j\left(m\times\HP\right) \perp
\hyp\hspace{1.8cm } \text{[by Proposition \ref{N1}]}\\
&=\lperp_{j \text{ even}}S^{i}q_S\ox \left({{m+\frac{j}{2}-1}\choose{\frac{j}{2}}}\times\<1\>
\perp \frac{1}{2}\left({{2m+j-1}\choose{j}}-{{m+\frac{j}{2}-1}\choose{\frac{j}{2}}}\right)\times \HP \right) \perp
\hyp\intertext{\hspace{8.3cm}[by Proposition \ref{N2}]}
&=\lperp_{j \text{ even}}S^{i}q_S\ox \left({{m+\frac{j}{2}-1}\choose{\frac{j}{2}}}\times\<1\>\right) 
\perp \hyp\\
&=\sum_{j \text{ even}}{{m+\frac{j}{2}-1}\choose{\frac{j}{2}}}\times S^i q_S 
\perp \hyp \\
&=\sum_{j \text{ even}}{{m+\frac{j}{2}-1}\choose{\frac{j}{2}}}\times 
\left(\lperp_{\ell=0}^{i/2}{{3+\ell}\choose{\ell}}\times \gL^{i-2\ell}q_S\right) 
\perp\hyp \hspace{1cm } \text{[by Proposition \ref{S3}]}\\
&=\sum_{j \text{ even}}{{m+\frac{j}{2}-1}\choose{\frac{j}{2}}}\times 
\left({{3+\frac{i}{2}}\choose{\frac{i}{2}}}\times \gL^{0}q_S 
\perp{{3+\frac{i-2}{2}}\choose{\frac{i-2}{2}}}\times \gL^{2}q_S 
\perp {{3+\frac{i-4}{2}}\choose{\frac{i-4}{2}}}\times \gL^{4}q_S 
  \right) 
\perp\hyp \hspace{1mm} \intertext{\hspace{8cm}[since for larger $k$, $\gL^k$ is zero]}
&=\sum_{j \text{ even}}{{m+\frac{j}{2}-1}\choose{\frac{j}{2}}}\times 
\left({{3+\frac{i}{2}}\choose{\frac{i}{2}}}\times \<1\> 
\perp{{3+\frac{i-2}{2}}\choose{\frac{i-2}{2}}}\times \<a,b,ab\>\ox\<1,(-1)^{n/2}\> 
\perp {{3+\frac{i-4}{2}}\choose{\frac{i-4}{2}}}\times \<(-1)^{n/2}\>
  \right) \\
  &\qquad\perp \hyp \\ 
&=\sum_{j \text{ even}}{{m+\frac{j}{2}-1}\choose{\frac{j}{2}}}
{{3+\frac{k-j}{2}}\choose{\frac{k-j}{2}}}\times \<1\> 
\perp \sum_{j \text{ even}}{{m+\frac{j}{2}-1}\choose{\frac{j}{2}}}{{3+\frac{k-j-2}{2}}\choose{\frac{k-j-2}{2}}}\times \<a,b,ab\>\ox\<1,(-1)^{n/2}\> \\
  &\qquad\perp \sum_{j \text{ even}}{{m+\frac{j}{2}-1}\choose{\frac{j}{2}}}{{3+\frac{k-j-4}{2}}\choose{\frac{k-j-4}{2}}}\times \<(-1)^{n/2}\>
  \perp \hyp \\
  &=\sum_{j=0}^{k/2}{{m+j'-1}\choose{j'}}
{{3+k/2-j'}\choose{k/2-j'}}\times \<1\> 
\perp \sum_{j'=0}^{k/2-1}{{m+j'-1}\choose{j'}}{{3+k/2-1-j'}\choose{k/2-1-j'}}\times \<a,b,ab\>\ox\<1,(-1)^{n/2}\> \\
  &\qquad\perp \sum_{j'=0}^{k/2-2}{{m+j'-1}\choose{j'}}{{3+k/2-2-j'}\choose{k/2-2-j'}}\times \<(-1)^{n/2}\>
  \perp \hyp \intertext{\hspace{9cm }[changing the index, $j'=j/2$]}
  &={{m+3+k/2}\choose{k/2}}\times \<1\> 
\perp {{m+2+k/2}\choose{k/2-1}}\times \<a,b,ab\>\ox\<1,(-1)^{n/2}\> \\
  &\qquad\perp {{m+1+k/2}\choose{k/2-2}}\times \<(-1)^{n/2}\>
  \perp \hyp\ . \hspace{1cm}\text{[by \cite[Equation (3)]{G}]}
\end{align*}
We now examine separately the cases of $n\equiv 0 \pmod 4$ and $n\equiv 2 \pmod 4$. 
Suppose $n\equiv 0 \pmod 4$. In this case we have $(-1)^{n/2}=1$ and so
\begin{align*}
 S^k{T_S}
  &=\left({{m+3+k/2}\choose{k/2}} + {{m+1+k/2}\choose{k/2-2}}\right)\times \<1\> 
\perp 2{{m+2+k/2}\choose{k/2-1}}\times \<a,b,ab\> \perp \hyp\\
  &=\left(1+\frac{(m+3+k/2)(m+2+k/2)}{(k/2)(k/2-1)}\right){{m+1+k/2}\choose{k/2-2}}\times \<1\> 
\perp 2{{m+2+k/2}\choose{k/2-1}}\times \<a,b,ab\> \perp \hyp\\ 
  &=\left(1+\frac{(m+3+k/2)(m+2+k/2)}{(k/2)(k/2-1)}\right){{m+1+k/2}\choose{k/2-2}}\times \<1\> 
\perp \hyp\ .
\end{align*} 

Now suppose $n\equiv 2 \pmod 4$. In this case we have $(-1)^{n/2}=-1$ and so
\begin{align*}
 S^k{T_S}
  &={{m+3+k/2}\choose{k/2}} \times \<1\> 
\perp {{m+1+k/2}\choose{k/2-2}}\times \<-1\> \perp \hyp \\
  &=\left({{m+3+k/2}\choose{k/2}} -
{{m+1+k/2}\choose{k/2-2}}\right)\times \<1\> \perp \hyp \\
  &=\left(\frac{(m+3+k/2)(m+2+k/2)}{(k/2)(k/2-1)}-1\right){{m+1+k/2}\choose{k/2-2}}\times \<1\> \perp \hyp\ .
\end{align*} 

\end{proof}

\section{Summary of results}

We give here a summary of the above results on symmetric powers alongside the corresponding results on exterior powers computed in \cite[Section 5]{F}.

Let $\vf\simeq h\times \HP$ for arbitrary $h\in \N$. Then

\begin{table}[h]
  \begin{center}
    \begin{tabular}{ | c || c | c |}
       \hline
       \bf{$k$} & \bf{$\gL^k \vf$}  & \bf{$S^k \vf$} \\ \hline \hline
       odd       & $\hyp$   & $\hyp$    \\ \hline
       even      & $\displaystyle{h\choose\frac{k}{2}}\times \<(-1)^\frac{k}{2}\>\perp\hyp$    & $\displaystyle{h+\frac{k}{2}-1\choose\frac{k}{2}}\times \<1\>\perp\hyp$   \\ \hline

    \end{tabular}
  \end{center}
\end{table}

Let $S=(a,b;n,K,\omega)$ be a symbol algebra of degree $n$ as described in Section \ref{sec.sa}. Let $\vf= T_S= \<n\>\<1, a, b, (-1)^\frac{n}{2}ab\>$ and $m=\nslfo$. Then

\begin{table}[h]
  \begin{center}
    \begin{tabular}{ | c | c || c | c |}
       \hline
       \bf{$n$} & \bf{$k$} & \bf{$\gL^k \vf$}  & \bf{$S^k \vf$} \\ \hline \hline
       odd & odd & ${\nslo \choose \frac{k-1}{2}} \times \<(-1)^\frac{k-1}{2}\> \perp \hyp$   & ${{\frac{n^2+k-2}{2}} \choose {\frac{k-1}{2}}} \times \<(-1)^\frac{n-1}{2}\> \perp \hyp$    \\ \hline
       odd & even & ${\nslo \choose \frac{k}{2}}\times \<(-1)^\frac{k}{2}\> \perp \hyp$   & ${{\frac{n^2+k-1}{2}} \choose \frac{k}{2}}\times \<1\> \perp \hyp$    \\ \hline
       $0\pmod 4$ & odd & ${\dbinom {m+1} {\frac{k-1}{2}}} \times \<(-1)^\frac{n(k-1)}{4}\> q_S \perp \hyp$    
       		& $\frac{2m+2k+4}{k-1} {{\frac{2m+k+3}{2}}\choose{\frac{k-3}{2}}}\times q_S \perp\hyp$   \\ \hline
       $0\pmod 4$ & even & ${\dbinom {\frac{n^2}{2}} {\frac{k}{2}}}\times \<1\>\perp \hyp$    
       		& $\left(1+\frac{(m+3+k/2)(m+2+k/2)}{(k/2)(k/2-1)}\right){{\frac{2m+2+k}{2}}\choose{\frac{k-4}{2}}}\times \<1\> \perp\hyp$   \\ \hline
       $2\pmod 4$ & odd & ${\dbinom {m+1} {\frac{k-1}{2}}} \times \<(-1)^\frac{n(k-1)}{4}\> q_S\perp \hyp$    
       		& $\frac{2m+6}{k-1} {{\frac{2m+k+3}{2}}\choose{\frac{k-3}{2}}}\times q_S \perp\hyp$   \\ \hline
       $2\pmod 4$ & even 
         & $
       			\begin{cases}
				\lb 1-\frac{2k}{n^2}\rb{\dbinom{\frac{n^2}{2}}{\frac{k}{2}}} \times \<(-1)^\frac{k}{2}\>
				\perp \hyp \text{, $k\le \frac{n^2}{2}$\kern-1em}\\
			\lb \frac{2k}{n^2}-1\rb{\dbinom{\frac{n^2}{2}}{\frac{k}{2}}} \times \<(-1)^\frac{k+2}{2}\>
				\perp \hyp \text{, $k > \frac{n^2}{2}$\kern-1em}
			\end{cases}$   
			 & $\left(\frac{(m +3+k/2)(m+2+k/2)}{(k/2)(k/2-1)}-1\right){{\frac{2m+2+k}{2}}\choose{\frac{k-4}{2}}}\times \<1\>
			 \perp\hyp$ 
                \\ \hline
    \end{tabular}
  \end{center}
\end{table}

\end{document}